\documentclass[pdflatex,sn-mathphys-num]{sn-jnl}

\newcommand{\hk}{\mathbf{HK}_{\Theta}}
\newcommand{\hkg}{\mathbf{HK}_{\Gamma}}
\newcommand{\vp}{\varphi}
\newcommand{\tot}{\overset{\Theta}{\to}}
\newcommand{\tog}{\overset{\Gamma}{\to}}
\newcommand{\sgt}{\widetilde{\mathcal{S}}_{\Gamma,\Theta}}
\newcommand{\stt}{\widetilde{\mathcal{S}}_{\Theta,\Theta}}

\usepackage[T1]{fontenc}
\usepackage[utf8]{inputenc}
\usepackage{bbm}
\usepackage{amsmath, amsthm, amssymb, amsfonts}
\usepackage{hyperref}
\usepackage{enumitem}
\usepackage{mathrsfs}
\usepackage{tikz}
\usepackage{graphicx}

\newtheorem{theorem}{Theorem}[section]
\newtheorem{lemma}[theorem]{Lemma}
\newtheorem{proposition}[theorem]{Proposition}
\newtheorem{corollary}[theorem]{Corollary}
\newtheorem*{claim}{Claim}
\theoremstyle{definition}
\newtheorem{definition}[theorem]{Definition}

\theoremstyle{remark}
\newtheorem{remark}[theorem]{Remark}

\setlist[itemize]{
  topsep=6pt,
  partopsep=0pt,
  itemsep=2pt,
  parsep=0pt,
  labelindent=\parindent,
  leftmargin=*,
  align=left
}

\setlist[enumerate,1]{
  topsep=6pt,
  partopsep=0pt,
  itemsep=2pt,
  parsep=0pt,
  labelindent=\parindent,
  leftmargin=*,
  align=right
}

\begin{document}

\title[Homomorphisms of Hecke--Kiselman Monoids Associated to Simple Oriented Graphs]{Homomorphisms of Hecke--Kiselman Monoids Associated to Simple Oriented Graphs}


\author*[1]{\fnm{Luka} \sur{Andren\v sek}}\email{andrensek.luka@gmail.com}

\affil*[1]{\orgdiv{Department of Mathematics}, \orgname{Faculty of Mathematics and Physics, University of Ljubljana}, \orgaddress{\street{Jadranska ulica 19}, \city{Ljubljana}, \postcode{SI-1000},  \country{Slovenia}}}

\abstract{
   \normalfont\unboldmath
   Let $\Gamma$ and $\Theta$ be finite simple oriented graphs, and let $\mathbf{HK}_{\Gamma}$ and $\mathbf{HK}_{\Theta}$ denote the corresponding Hecke--Kiselman monoids.
   We present a set of Boolean matrices that is in bijection with the set of monoid homomorphisms from $\mathbf{HK}_{\Gamma}$ to $\mathbf{HK}_{\Theta}$.
   In particular, we obtain a Boolean matrix monoid isomorphic to the endomorphism monoid of $\mathbf{HK}_{\Theta}$.
}


\maketitle

\section{Introduction}

Let $\Theta$ be a finite simple oriented graph, that is, a finite oriented graph with no multiple arrows, no loops, and no oriented cycles of length $2$.
Let $V_{\Theta} = \{x_1, x_2, \dots, x_n\}$ be the vertex set of $\Theta$ for some positive integer $n$.
If there is an arrow from $x_i$ to $x_j$ in $\Theta$, we write $x_i \tot x_j$.
For distinct vertices $x_i$ and $x_j$, we say that $x_i$ and $x_j$ are not connected in $\Theta$ if there is no arrow from $x_i$ to $x_j$ and there is no arrow from $x_j$ to $x_i$ in $\Theta$.

In~\cite{ganyushkin11}, Ganyushkin and Mazorchuk define the \emph{Hecke--Kiselman monoid} $\hk$ associated to $\Theta$ as the monoid generated by
\[
x_1, x_2, \dots, x_n
\]
subject to the following \emph{defining relations} for all $x_i, x_j \in V_{\Theta}$:
\begin{itemize}
    \item $x_i^2 = x_i$;
    \item $x_i x_j = x_j x_i$ if $x_i$ and $x_j$ are not connected in $\Theta$;
    \item $x_i x_j x_i = x_j x_i x_j = x_i x_j$ if $x_i \tot x_j$.
\end{itemize}
We denote by $e$ the unit element of $\hk$.
Let $\Gamma$ be another finite simple oriented graph with vertex set $V_{\Gamma} = \{y_1, y_2, \dots, y_m\}$ for some positive integer $m$,
and let $\hkg$ denote the corresponding Hecke--Kiselman monoid.

This paper studies the set of all monoid homomorphisms from $\hkg$ to $\hk$, which we denote by $\mathrm{Hom}(\hkg, \hk)$, and the endomorphism monoid $\mathrm{End}(\hk)$ whose operation is the usual composition of maps.
Throughout, monoid homomorphisms are assumed to preserve the identity element.

Hecke--Kiselman monoids were originally introduced in the setting of finite simple digraphs~\cite{ganyushkin11}.
They are a natural mixture of two families of monoids, namely, 0-Hecke monoids, which are closely connected to Coxeter groups, and Kiselman monoids~\cite{kiselman,kudryavtseva}.
They have been studied by various authors~\cite{aragona13, aragona20, Ashikhmin15, collina, forsberg2017, ganyushkin11, Lebed2025, wiertel24}.
Particular interest has been devoted to the corresponding Hecke--Kiselman algebras~\cite{okninski20,okninski21,wiertel22,wiertel23}.
In this paper, we only consider Hecke--Kiselman monoids associated to simple oriented graphs.

From the proof of~\cite[Theorem~16]{ganyushkin11}, it follows that monoid isomorphisms from $\hkg$ to $\hk$ are in bijection with graph isomorphisms from $\Gamma$ to $\Theta$.
Setting $\Gamma = \Theta$, it can be verified that the automorphism group of $\hk$ is isomorphic to the automorphism group of the graph $\Theta$.
To the best of our knowledge, general monoid homomorphisms from $\hkg$ to $\hk$ and the structure of the endomorphism monoid of $\hk$ have not been systematically studied in the literature.
However, a description of endomorphisms of the Kiselman monoid, which is the Hecke--Kiselman monoid corresponding to the complete acyclic oriented graph, is known in the literature~\cite{andrensek}.

The main result of this paper is Theorem~\ref{hom:main}.
We show that $\mathrm{Hom}(\hkg, \hk)$ is in bijection with the set of Boolean matrices corresponding to $(\Gamma,\Theta)$\emph{-admissible} sequences of sets,
which are defined in Definition~\ref{definition:p-pure}.
These are finite sequences of length $m$ of subsets of $V_{\Theta}$ for which the induced full subgraphs of $\Theta$ do not contain any oriented cycles,
and which satisfy a simple graph-theoretic condition, later denoted by $p_{\Theta}$ and defined in Definition~\ref{definition:p}, according to arrows in $\Gamma$.
Furthermore, in Theorem~\ref{endo:main}, we show that in the case $\Gamma = \Theta$, the constructed bijection between $\mathrm{End}(\hk)$ and the corresponding Boolean matrices is a monoid isomorphism.

The paper is organised as follows.
In Section~\ref{PRELIM}, we establish basic notation and recall results from the relevant literature.
In Section~\ref{PROP}, we prove new results on idempotents in $\hk$.
These results are used to prove the main result in Section~\ref{HOM},
and we consider the endomorphism monoid of $\hk$ in Section~\ref{ENDO}.

\section{Preliminaries}\label{PRELIM}

Denote by $F(V_{\Theta})$ the free monoid over the alphabet $V_{\Theta} = \{x_1, x_2, \dots, x_n\}$.
Elements of $F(V_{\Theta})$ are called \emph{words}, and each $x_i \in V_{\Theta}$ is called a \emph{letter}.
In what follows, we will use the term \emph{letter} to refer to the vertices of $\Theta$, generators of $\hk$, and generators of $F(V_{\Theta})$.
We have the \emph{canonical projection}
\[
\pi : F(V_{\Theta}) \twoheadrightarrow \hk.
\]
For a word $w \in F(V_{\Theta})$, denote by $\ell(w)$ its length.
The unique word of length $0$ is called the \emph{empty word}.

\subsection{The Content Map}

For a word $w \in F(V_{\Theta})$, we define its content $c(w)$ as the set of $x_i \in V_{\Theta}$ for which $x_i$ appears in $w$.
Since the defining relations of $\hk$ preserve content, we can define the content and denote it again by $c(x)$ for $x \in \hk$ as the content of any word $w \in F(V_{\Theta})$ with $x = \pi(w)$.
Moreover, we have 
\[
c(xy) = c(x) \cup c(y) \quad \text{and} \quad c(e) = \emptyset,
\]
which shows that $c : \hk \twoheadrightarrow \mathbbm{2}^{V_{\Theta}}$ is a monoid epimorphism, where $\mathbbm{2}^{V_{\Theta}}$ is equipped with the operation $\cup$.

We use the following lemma in the proof of Proposition~\ref{phi:homo}.
\begin{lemma}\label{content:lemma}
   Let $\psi, \vp \in \mathrm{End}(\hk)$ and $x_i \in \hk$.
   Then 
   \[
   c((\psi \circ \vp)(x_i)) = \bigcup_{x_j \in c(\vp(x_i))} c(\psi(x_j)).
   \]
\end{lemma}

\begin{proof}
   If $\vp(x_i) = e$, then we have $c(\vp(x_i)) = \emptyset$ and $(\psi \circ \vp)(x_i) = e$.
   Hence $c((\psi \circ \vp)(x_i)) = \emptyset$, and thus the conclusion holds in this case.
   
   Now suppose $\vp(x_i) \neq e$ and write $\vp(x_i) = x_{i_1} \dots x_{i_k}$ for some letters $x_{i_1}, \dots, x_{i_k} \in \hk$ and $k \geq 1$.
   Then $c(\vp(x_i)) = \{x_{i_1}, \dots, x_{i_k}\}$ and we further have
   \[
   (\psi \circ \vp)(x_i) = \psi(\vp(x_i)) = \psi(x_{i_1} \dots x_{i_k}) = \psi(x_{i_1}) \dots \psi(x_{i_k}).
   \]
   Hence we have 
   \[
   c((\psi \circ \vp)(x_i)) = c(\psi(x_{i_1})) \cup \dots \cup c(\psi(x_{i_k})),
   \]
   and therefore, 
   \[
   c((\psi \circ \vp)(x_i)) = \bigcup_{x_j \in \{x_{i_1}, \dots, x_{i_k}\}} c(\psi(x_j)) = \bigcup_{x_j \in c(\vp(x_i))} c(\psi(x_j)).
   \]
\end{proof}

\subsection{Normal Forms in $\hk$}\label{NF}

Here we recall the main results from~\cite{aragona20}.
The following lemma is proved in~\cite[Lemma~1]{aragona20}.
\begin{lemma}\label{NF:lemma1}
   If $x_i \in \hk$ is a letter and $w \in \hk$ is obtained by multiplying letters that do not admit arrows to (respectively from) $x_i$,
   then $x_i w x_i = x_i w$ (respectively $x_i w x_i = w x_i$).
\end{lemma}

The \emph{right cancellation} $\overset{r}{\to}$ and \emph{left cancellation} $\overset{l}{\to}$ are binary relations on $F(V_{\Theta})$ defined as follows.
For $w, w' \in F(V_{\Theta})$, we have
\begin{itemize}
   \item $w \overset{r}{\to} w'$ if there exist words $w_1, w_2, u \in F(V_{\Theta})$, and a letter $x_i \in F(V_{\Theta})$ such that 
   $w = w_1 x_i u x_i w_2$, $w' = w_1 x_i u w_2$, and no letter in $u$ has an arrow to $x_i$.
   \item $w \overset{l}{\to} w'$ if there exist words $w_1, w_2, u \in F(V_{\Theta})$, and a letter $x_i \in F(V_{\Theta})$ such that 
   $w = w_1 x_i u x_i w_2$, $w' = w_1 u x_i w_2$, and no letter in $u$ has an arrow from $x_i$.
\end{itemize}
Right and left cancellations are also called \emph{elementary cancellations}.
Denote by $\to = \overset{r}{\to} \cup \overset{l}{\to}$ and let $\overset{*}{\to}$ be the reflexive-transitive closure of $\to$ on $F(V_{\Theta})$.
Then Lemma~\ref{NF:lemma1} implies that for $u, v \in F(V_{\Theta})$, if $u \to v$, then $\pi(u) = \pi(v)$.
Therefore, if $w \overset{*}{\to} w'$ for $w, w' \in F(V_{\Theta})$, then $\pi(w) = \pi(w')$.

\begin{definition}
   A word $w \in F(V_{\Theta})$ is a \emph{normal form} of $x \in \hk$ if $x = \pi(w)$ and no elementary cancellations may be performed on $w$.
\end{definition}
Every $x \in \hk$ has at least one normal form, as any word of minimal length in $\pi^{-1}(\{x\})$ must be a normal form of $x$.
Denote by $\mathcal{N}$ the set of all normal forms of all elements of $\hk$.

Define the equivalence relation $\sim$ on $F(V_{\Theta})$, which is generated by \emph{elementary commutations}:
\[
w_1 x_i x_j w_2 \sim w_1 x_j x_i w_2,
\]
for all words $w_1, w_2 \in F(V_{\Theta})$, and letters $x_i, x_j \in F(V_{\Theta})$ such that $x_i$ and $x_j$ are not connected in $\Theta$.
The following was proved in~\cite[Theorem 1]{aragona20}.
\begin{theorem}\label{normal:form}
   The following statements hold.
   \begin{enumerate}[label=(\roman*), ref=\roman*]
      \item\label{normal:form:i} Let $w, w' \in F(V_{\Theta})$. If $\pi(w) = \pi(w')$ and $u, v \in \mathcal{N}$ satisfy
         $w \overset{*}{\to} u$ and $w' \overset{*}{\to} v$, then $u \sim v$.
      \item\label{normal:form:ii} Every simplifying sequence 
      \[
      w \to w_1 \to w_2 \to \dots \to w_k
      \]
      may be extended to a simplifying sequence ending on a normal form of $\pi(w)$.
   \end{enumerate}
\end{theorem}

\subsection{Idempotents in $\hk$}

For $X \subseteq V_{\Theta}$, denote by $\Theta[X]$ the full subgraph of $\Theta$ with vertex set $X$, that is,
$\Theta[X]$ is an oriented graph with vertex set $X$ and for $x_i, x_j \in X$, there is an arrow from $x_i$ to $x_j$ in $\Theta[X]$ if and only if there is an arrow from $x_i$ to $x_j$ in $\Theta$.

In~\cite[Theorem 3.6]{forsberg2017}, the following theorem is proved for a general finite simple oriented graph $\Theta$.
\begin{theorem}\label{theorem:forsberg}
   There is a bijection between idempotents in $\hk$ and full subgraphs of $\Theta$ which do not contain any oriented cycles.
\end{theorem}
In the proof, an idempotent is mapped to the full subgraph of $\Theta$ whose vertex set is the content of the idempotent.
Therefore, for each $X \subseteq V_{\Theta}$ such that $\Theta[X]$ does not contain any oriented cycles, there exists exactly one idempotent in $\hk$ with content $X$.
We denote this idempotent by $e_X$.

We denote the following set
\[
\mathcal{A}_{\Theta} = \{X \subseteq V_{\Theta} \mid \Theta[X] \text{ does not contain any oriented cycles}\}.
\]

\begin{remark}\label{A-theta}
   If $Y \subseteq X \in \mathcal{A}_{\Theta}$, then $\Theta[Y]$ is a subgraph of $\Theta[X]$, and hence it does not contain any oriented cycles.
   Therefore, $Y \in \mathcal{A}_{\Theta}$.
\end{remark}

Theorem~\ref{theorem:forsberg} implies the following proposition.
\begin{proposition}\label{proposition:idempotents}
   The set 
   \[
   \{e_X \mid X \in \mathcal{A}_{\Theta} \}
   \]
   is the set of all idempotents in $\hk$.
\end{proposition}

\begin{proposition}\label{hk_idempotents}
    Let $X, Y \in \mathcal{A}_{\Theta}$.
    Then $e_X e_Y$ is an idempotent in $\hk$ if and only if $X \cup Y \in \mathcal{A}_{\Theta}$ and $e_X e_Y = e_{X \cup Y}$.
\end{proposition}

\begin{proof}
    If $X \cup Y \in \mathcal{A}_{\Theta}$ and $e_X e_Y = e_{X \cup Y}$, then $e_X e_Y$ is an idempotent by Proposition~\ref{proposition:idempotents}.
    Conversely, if $e_X e_Y$ is an idempotent, then by Proposition~\ref{proposition:idempotents}, there exists $Z \in \mathcal{A}_{\Theta}$ such that 
    \[
    e_X e_Y = e_Z.
    \]
    Taking contents, we get 
    \[
    Z = c(e_Z) = c(e_X e_Y) = c(e_X) \cup c(e_Y) = X \cup Y. 
    \]
    Hence $X \cup Y \in \mathcal{A}_{\Theta}$ and $e_X e_Y = e_{X \cup Y}$.
\end{proof}

The idempotent $e_X$ can be written explicitly.
Note that if $\Theta[X]$ does not contain any oriented cycles, then we can write $X$ as $X = \{x_{i_1}, x_{i_2}, \dots, x_{i_k}\}$ such that
$x_{i_j} \tot x_{i_{j'}}$ implies $j < j'$ for all $j, j' \in \{1, 2, \dots, k\}$.

\begin{lemma}\label{idempotent:formula}
   Let $X \in \mathcal{A}_{\Theta}$.
   Then for any ordering of the set $X = \{x_{i_1}, x_{i_2}, \dots, x_{i_k}\}$ such that $x_{i_j} \tot x_{i_{j'}}$ implies $j < j'$ for all $j, j' \in \{1, 2, \dots, k\}$, we have
   \[
   e_X = x_{i_1} x_{i_2} \dots x_{i_k}.
   \]
   Moreover, we have $e_X z = z e_X = e_X$ for any $z \in \hk$ with $c(z) \subseteq X$.
\end{lemma}

\begin{proof}
   Denote by 
   \[
   y = x_{i_1} x_{i_2} \dots x_{i_k}.
   \]
   If $x_{i_j} \in X$, then
   \[
   x_{i_{j+1}} x_{i_{j+2}} \dots x_{i_k}
   \]
   is obtained by multiplying letters that do not admit arrows to $x_{i_j}$.
   Lemma~\ref{NF:lemma1} implies
   \[
   y x_{i_j} = x_{i_1} \dots x_{i_{j-1}} (x_{i_j} x_{i_{j+1}} \dots x_{i_k} x_{i_j}) = x_{i_1} \dots x_{i_{j-1}} (x_{i_j} x_{i_{j+1}} \dots x_{i_k}) = y.
   \]
   By induction, we get 
   \[
   y x_{l_1} x_{l_2} \dots x_{l_m} = y,
   \]
   for any $x_{l_1}, x_{l_2}, \dots, x_{l_m} \in X$.
   This proves $yz = y$ for any $z \in \hk$ with $c(z) \subseteq X$.

   For any $x_{i_j} \in X$,
   \[
   x_{i_1} x_{i_2} \dots x_{i_{j-1}}
   \]
   is obtained by multiplying letters that do not admit arrows from $x_{i_j}$.
   Lemma~\ref{NF:lemma1} implies
   \[
   x_{i_j} y = (x_{i_j} x_{i_1} \dots x_{i_{j-1}} x_{i_j}) x_{i_{j+1}} \dots x_{i_k} = (x_{i_1} \dots x_{i_{j-1}} x_{i_j}) x_{i_{j+1}} \dots x_{i_k} = y.
   \]
   By induction, we again obtain 
   \[
   x_{l_1} x_{l_2} \dots x_{l_m} y = y,
   \]
   for any $x_{l_1}, x_{l_2}, \dots, x_{l_m} \in X$, which proves $zy = y$ for any $z \in \hk$ with $c(z) \subseteq X$.
   Since $c(y) = X$, we have
   \[
   y^2 = y y = y.
   \]
   Hence $y$ is an idempotent in $\hk$.
   Since $c(y) = X$, it follows that $y = e_X$.
\end{proof}

\begin{definition}
   A word $w \in F(V_{\Theta})$ is said to \emph{have no arrows from right to left} if 
   \[
   w = x_{i_1} x_{i_2} \dots x_{i_k},
   \]
   where $x_{i_1}, x_{i_2}, \dots, x_{i_k} \in F(V_{\Theta})$ are letters, and if there is an arrow from $x_{i_j}$ to $x_{i_{j'}}$ then $j < j'$ for all $j, j' \in \{1, 2, \dots, k\}$.

   Moreover, for $X \in \mathcal{A}_{\Theta}$, $w$ is said to be \emph{$X$-topological} if $w$ has no arrows from right to left, $c(w) = X$, and $\ell(w) = |X|$.
   In other words, $w$ is $X$-topological if there are no arrows from right to left and each letter $x_i$ with $x_i \in X$ occurs in $w$ exactly once.
\end{definition}

\begin{remark}\label{remark-top}
   If $X \in \mathcal{A}_{\Theta}$, then $\Theta[X]$ is an acyclic oriented graph and hence there exists an $X$-topological word.
   Also observe that the empty word is $\emptyset$-topological.
\end{remark}

\begin{lemma}\label{X-top}
   For $X \in \mathcal{A}_{\Theta}$, a word $w \in F(V_{\Theta})$ is $X$-topological if and only if $w$ is a normal form of $e_X$.
\end{lemma}

\begin{proof}
   If $w$ is $X$-topological, then $\pi(w) = e_X$ by Lemma~\ref{idempotent:formula}.
   Moreover, since no letter in $w$ occurs more than once, no further elementary cancellations may be performed on $w$, and hence $w$ is a normal form of $e_X$.

   Now assume $w$ is a normal form of $e_X$.
   Since $X \in \mathcal{A}_{\Theta}$, there exists an $X$-topological word $v$ by Remark~\ref{remark-top}.
   Hence, there are no arrows from right to left in $v$, $c(v) = X$, and $\ell(v) = |X|$.
   By the previous paragraph, $v$ is a normal form of $e_X$.
   By Theorem~\ref{normal:form}~(\ref{normal:form:i}), we have $w \sim v$.
   Since elementary commutations only interchange letters that are not connected, the relative order of any two letters that are connected by an arrow is preserved under $\sim$.
   Hence, since $v$ has no arrows from right to left, the same is true for $w$.
   Since $\sim$ also preserves content and length, we have $c(w) = c(v) = X$ and $\ell(w) = \ell(v) = |X|$.
   This implies that $w$ is $X$-topological.
\end{proof}

\begin{lemma}\label{IDS:no_arrows}
   Let $X, Y \in \mathcal{A}_{\Theta}$.
   If there are no arrows starting in $Y$ and ending in $X$ and $X \cap Y = \emptyset$, then $X \cup Y \in \mathcal{A}_{\Theta}$ and
   \[
   e_X e_Y = e_{X \cup Y}.
   \]
\end{lemma}

\begin{proof}
   Since $X \cap Y = \emptyset$ and there are no arrows starting in $Y$ and ending in $X$, any oriented cycle in $\Theta[X \cup Y]$ must entirely lie in either $X$ or $Y$.
   Since $X, Y \in \mathcal{A}_{\Theta}$, it follows that there are no oriented cycles in $\Theta[X \cup Y]$, and hence $X \cup Y \in \mathcal{A}_{\Theta}$.
   Let $w_X$ be a normal form of $e_X$ and let $w_Y$ be a normal form of $e_Y$.
   Lemma~\ref{X-top} implies that $w_X$ is $X$-topological and $w_Y$ is $Y$-topological.
   Since there are no arrows from right to left in $w_X$ and $w_Y$, respectively, and there are no arrows starting in $Y$ and ending in $X$, there are no arrows from right to left in the word $w_X w_Y$.
   Moreover, we have
   \[
   c(w_X w_Y) = c(w_X) \cup c(w_Y) = X \cup Y,
   \]
   and 
   \[
   \ell(w_X w_Y) = \ell(w_X) + \ell(w_Y) = |X| + |Y| = |X \cup Y|.
   \]
   Hence $w_X w_Y$ is $X \cup Y$-topological.
   Therefore, Lemma~\ref{X-top} implies that $w_X w_Y$ is a normal form of $e_{X \cup Y}$, and hence
   \[
   e_X e_Y = \pi(w_X) \pi(w_Y) = \pi(w_X w_Y) = e_{X \cup Y}.
   \]
\end{proof}

\section{Properties of Idempotents in $\hk$}\label{PROP}

In this section, we give some additional properties of idempotents in $\hk$, which are used in Section~\ref{HOM}.

\begin{definition}\label{definition:p}
   For $X, Y \in \mathcal{A}_{\Theta}$, define the predicate $p_{\Theta}$ as follows.
   Condition $p_{\Theta}(X, Y)$ holds if and only if there do not exist $x_b \in Y \setminus X$, $x_a \in X \setminus Y$, $r \in \mathbb{N}_0$, and $x_{i_1}, \dots, x_{i_r} \in X \cap Y$ such that 
   \[
   x_b \tot x_{i_1} \tot \dots \tot x_{i_r} \tot x_a.
   \]
\end{definition}

\begin{remark}
   If $r = 0$, a chain $x_b \tot x_{i_1} \tot \dots \tot x_{i_r} \tot x_a$ is interpreted as $x_b \tot x_a$.
\end{remark}

\begin{lemma}
   Let $X, Y \in \mathcal{A}_{\Theta}$.
   Then $p_{\Theta}(X, Y)$ holds if and only if there does not exist a directed path in $\Theta[X \cup Y]$ whose startpoint lies in $Y \setminus X$ and whose endpoint lies in $X \setminus Y$.
\end{lemma}

\begin{proof}
   If there is no directed path in $\Theta[X \cup Y]$ which starts in some vertex of $Y \setminus X$ and ends in some vertex of $X \setminus Y$, then $p_{\Theta}(X, Y)$ must hold.
   Assume $p_{\Theta}(X, Y)$ holds.
   Suppose there exists a directed path
   \[
   x_{i_1} \tot x_{i_2} \tot \dots \tot x_{i_k},
   \]
   such that $x_{i_1} \in Y \setminus X$, $x_{i_k} \in X \setminus Y$, and $x_{i_1}, x_{i_2}, \dots, x_{i_k} \in X \cup Y$.
   Let
   \[
   j_1 = \max \{j \in \{1, \dots, k\} \mid x_{i_j} \in Y \setminus X \} \quad \text{and} \quad j_2 = \min \{j \in \{j_1 + 1, \dots, k\} \mid x_{i_j} \in X \setminus Y \}.
   \]
   Since $x_{i_k} \in X \setminus Y$ and $j_1 < k$, the index $j_2$ is well-defined, and $j_1 < j_2$.
   Then the directed path
   \[
   x_{i_{j_1}} \tot x_{i_{j_1 + 1}} \tot \dots \tot x_{i_{j_2 - 1}} \tot x_{i_{j_2}} 
   \]
   starts in $Y \setminus X$ and ends in $X \setminus Y$.
   Moreover, the vertices $x_{i_{j_1 + 1}}, \dots, x_{i_{j_2 - 1}}$ lie in $X \cap Y$; by the maximality of $j_1$, none of them lies in $Y \setminus X$, and by the minimality of $j_2$, none of them lies in $X \setminus Y$.
   Since all vertices of the directed path lie in $X \cup Y$, the vertices $x_{i_{j_1 + 1}}, \dots, x_{i_{j_2 - 1}}$ must lie in $X \cap Y$.
   This contradicts the fact that $p_{\Theta}(X, Y)$ holds and concludes the proof.
\end{proof}

Let $w = x_{i_1} x_{i_2} \dots x_{i_k} \in F(V_{\Theta})$ be a word, where $x_{i_1}, x_{i_2}, \dots, x_{i_k} \in F(V_{\Theta})$ are letters.
A word of the form $u = x_{i_{j_1}} x_{i_{j_2}} \dots x_{i_{j_l}}$ with $1 \leq j_1 < j_2 < \dots < j_l \leq k$ is called a \emph{quasi-subword} of $w$.
Moreover, the empty word is a quasi-subword of any word.
From the definition of $\to$, it follows that if $u \overset{*}{\to} v$, then $v$ is a quasi-subword of $u$.

\begin{proposition}\label{prop1}
   Let $X, Y \in \mathcal{A}_{\Theta}$.
   If $e_X e_Y$ is an idempotent in $\hk$, then $p_{\Theta}(X, Y)$ holds.
\end{proposition}

\begin{proof}
   Assume $p_{\Theta}(X, Y)$ does not hold.
   We show that $e_X e_Y$ is not an idempotent.
   Suppose that, on the contrary, $e_X e_Y$ is an idempotent.
   We now construct a contradiction.
   Let $w_X$ be a normal form of $e_X$ and let $w_Y$ be a normal form of $e_Y$.
   Lemma~\ref{X-top} implies that $w_X$ is $X$-topological and $w_Y$ is $Y$-topological.
   We also have $\pi(w_X w_Y) = e_X e_Y$.
   Since $e_X e_Y$ is an idempotent, by Proposition~\ref{hk_idempotents}, we have 
   \[
   X \cup Y \in \mathcal{A}_{\Theta} \quad \text{and} \quad e_X e_Y = e_{X \cup Y}.
   \]
   Let $u \in F(V_{\Theta})$ be a normal form of $e_{X \cup Y}$ with $w_X w_Y \overset{*}{\to} u$.
   Such a word $u$ exists by Theorem~\ref{normal:form}~(\ref{normal:form:ii}) and $\pi(w_X w_Y) = e_X e_Y = e_{X \cup Y}$.
   By Lemma~\ref{X-top}, $u$ is $X \cup Y$-topological.

   Since $p_{\Theta}(X, Y)$ does not hold, there exist $x_b \in Y \setminus X$, $x_a \in X \setminus Y$, $r \in \mathbb{N}_0$, and $x_{i_1}, \dots, x_{i_r} \in X \cap Y$ such that 
   \[
   x_b \tot x_{i_1} \tot \dots \tot x_{i_r} \tot x_a.
   \]
   Observe that each of the letters $x_a$ and $x_b$ occurs exactly once in $w_X w_Y$, and hence they occur exactly once in $u$, since an elementary cancellation cannot delete a letter that occurs only once.
   Moreover, $x_a$ occurs to the left of $x_b$ in $w_X w_Y$, since $x_a$ occurs in $w_X$ and $x_b$ occurs in $w_Y$.
   Since $w_X w_Y \overset{*}{\to} u$, $u$ is a quasi-subword of $w_X w_Y$, and hence $x_a$ occurs to the left of $x_b$ in $u$ as well.
   
   If $r = 0$, then $x_b \tot x_a$.
   Hence $x_a$ must occur to the right of $x_b$ in $u$, since there are no arrows from right to left in $u$, as $u$ is $X \cup Y$-topological.
   This contradicts the last conclusion of the previous paragraph, which established that $x_a$ occurs to the left of $x_b$ in $u$.
   Now assume $r > 0$.
   Since $u$ is $X \cup Y$-topological and $x_{i_1}, \dots, x_{i_r} \in X \cap Y \subseteq X \cup Y$, each of the letters $x_{i_1}, \dots, x_{i_r}$ occurs exactly once in $u$.
   Since $x_b \tot x_{i_1}$ and $u$ has no arrows from right to left, $x_{i_1}$ occurs to the right of $x_b$ in $u$.
   It follows that for $1 < t \leq r$, $x_{i_t}$ occurs to the right of $x_{i_{t-1}}$ in $u$, since $x_{i_{t-1}} \tot x_{i_t}$ and $u$ has no arrows from right to left.
   Inductively, we conclude that each $x_{i_t}$ for $1 \leq t \leq r$ occurs to the right of $x_b$ in $u$.
   Finally, it follows from $x_{i_r} \tot x_a$ and the fact that $u$ has no arrows from right to left that $x_a$ occurs to the right of $x_{i_r}$ in $u$, and hence $x_a$ occurs to the right of $x_b$ in $u$.
   In the previous paragraph, however, it is established that $x_a$ occurs to the left of $x_b$ in $u$.
   Hence we have reached a contradiction, and the proof is complete.
\end{proof}

\begin{proposition}\label{prop2}
   Let $X, Y \in \mathcal{A}_{\Theta}$.
   If $p_{\Theta}(X, Y)$ holds, then $e_X e_Y$ is an idempotent in $\hk$.
\end{proposition}

\begin{proof}
   Assume $p_{\Theta}(X, Y)$ holds.
   To prove that $e_X e_Y$ is an idempotent, by Proposition~\ref{hk_idempotents}, we need to show $X \cup Y \in \mathcal{A}_{\Theta}$ and $e_X e_Y = e_{X \cup Y}$.
   
   Denote $I = X \cap Y$ and
   \[
   R = \{x_i \in I \mid \exists x_b \in Y \setminus X, r \in \mathbb{N}_0,
   x_{i_1}, \dots, x_{i_r} \in I
   : x_b \tot x_{i_1} \tot \dots \tot x_{i_r} \tot x_i\}.
   \]
   Also set $S = I \setminus R$.
   We denote by
   \begin{equation}\label{M-N}
      M = (Y \setminus X) \cup R \quad \text{and} \quad N = (X \setminus Y) \cup S. 
   \end{equation}
   Observe that since $R, S \subseteq I = X \cap Y$, the above unions are disjoint.
   Then we further have 
   \begin{equation}\label{xy_decomp}
      X = N \cup R, \quad Y = M \cup S, \quad X \cup Y = N \cup M.
   \end{equation}
   The unions on the right-hand sides of these equalities are disjoint as well.
   The key observation is the following claim.
   \begin{claim}
      There are no arrows starting in $M$ and ending in $N$.
   \end{claim}

   \begin{proof}[Proof of the claim]
      We argue by contradiction.
      Let $x_i \in M$, $x_j \in N$ and suppose $x_i \tot x_j$.
      Now we consider cases on $x_i$ and $x_j$ based on~\eqref{M-N}.
      Since $p_{\Theta}(X, Y)$ holds, we cannot have $x_i \in Y \setminus X$ and $x_j \in X \setminus Y$.
      If $x_i \in Y \setminus X$ and $x_j \in S$, this implies $x_j \in R$ by the definition of $R$, which is a contradiction since $R \cap S = \emptyset$.
      If $x_i \in R$ and $x_j \in X \setminus Y$, then by the definition of $R$, there exist $x_b \in Y \setminus X$, $r \in \mathbb{N}_0$, and $x_{i_1}, \dots, x_{i_r} \in I$ such that
      \[
      x_b \tot x_{i_1} \tot \dots \tot x_{i_r} \tot x_i \tot x_j.
      \]
      Since $x_i \in R \subseteq I$ and $x_j \in X \setminus Y$, this contradicts the fact that $p_{\Theta}(X, Y)$ holds.
      If $x_i \in R$ and $x_j \in S$, we have some $x_b \in Y \setminus X$, $r \in \mathbb{N}_0$, and $x_{i_1}, \dots, x_{i_r} \in I$ such that 
      \[
      x_b \tot x_{i_1} \tot \dots \tot x_{i_r} \tot x_i \tot x_j.
      \]
      Since $x_i \in R \subseteq I$, this implies that $x_j \in R$.
      Thus $x_j \in R \cap S$, which is again a contradiction.
      The above cases are exhaustive by~\eqref{M-N}.
      Hence, in all cases, we have reached a contradiction, and therefore there are no arrows starting in $M$ and ending in $N$. 
   \end{proof}

   Since $N \subseteq X \in \mathcal{A}_{\Theta}$ and 
   $M \subseteq Y \in \mathcal{A}_{\Theta}$, Remark~\ref{A-theta} implies $N, M \in \mathcal{A}_{\Theta}$.
   Since $N \cap M = \emptyset$ and there are no arrows starting in $M$ and ending in $N$, Lemma~\ref{IDS:no_arrows} yields 
   \[
   N \cup M \in \mathcal{A}_{\Theta} \quad \text{and} \quad e_N e_M = e_{N \cup M}.
   \]
   Since $X \cup Y = N \cup M$ by \eqref{xy_decomp}, we obtain 
   \begin{equation}\label{t1}
      X \cup Y \in \mathcal{A}_{\Theta} \quad \text{and} \quad e_N e_M = e_{X \cup Y}.
   \end{equation}
   
   Since $R, S \subseteq X \cup Y \in \mathcal{A}_{\Theta}$, we also have $R, S \in \mathcal{A}_{\Theta}$ by Remark~\ref{A-theta}.
   As $R \subseteq M$ and there are no arrows starting in $M$ and ending in $N$, there are no arrows starting in $R$ and ending in $N$.
   Since $N \cap R = \emptyset$, Lemma~\ref{IDS:no_arrows} implies 
   \begin{equation}\label{t2}
      e_N e_R = e_{N \cup R} = e_X.
   \end{equation}
   Similarly, since $S \subseteq N$, there are no arrows starting in $M$ and ending in $S$.
   Since $S \cap M = \emptyset$, Lemma~\ref{IDS:no_arrows} implies 
   \begin{equation}\label{t3}
      e_S e_M = e_{S \cup M} = e_Y.
   \end{equation}
   Since $c(e_S) = S \subseteq X$ and $c(e_R) = R \subseteq M$, Lemma~\ref{idempotent:formula} implies 
   \begin{equation}\label{t4}
      e_X e_S = e_X \quad \text{and} \quad e_R e_M = e_M.
   \end{equation}
   Using \eqref{t1}, \eqref{t2}, \eqref{t3}, and \eqref{t4}, we obtain
   \[
   e_X e_Y = e_X e_S e_M = e_X e_M = e_N e_R e_M = e_N e_M = e_{X \cup Y}.
   \]
   This concludes the proof.
\end{proof}

Propositions~\ref{prop1} and~\ref{prop2} imply the following theorem.
\begin{theorem}\label{ide-p}
   Let $X, Y \in \mathcal{A}_{\Theta}$.
   Then $e_X e_Y$ is an idempotent in $\hk$ if and only if $p_{\Theta}(X, Y)$ holds.
\end{theorem}

\begin{lemma}\label{lemma:join}
   Let $X, Y \in \mathcal{A}_{\Theta}$ and assume that $X \cup Y \in \mathcal{A}_{\Theta}$ as well.
   Then $e_{X \cup Y} e_X$ and $e_Y e_{X \cup Y}$ are idempotents in $\hk$.
\end{lemma}

\begin{proof}
   Since $X \setminus (X \cup Y) = \emptyset$ and $Y \setminus (X \cup Y) = \emptyset$, conditions $p_{\Theta}(X \cup Y, X)$ and $p_{\Theta}(Y, X \cup Y)$ vacuously hold.
   The conclusion follows from Theorem~\ref{ide-p}.
\end{proof}

\begin{theorem}\label{theorem:p}
   Let $X, Y \in \mathcal{A}_{\Theta}$.
   Then $p_{\Theta}(X, Y)$ holds if and only if
   \[
   e_X e_Y e_X = e_Y e_X e_Y = e_X e_Y.
   \]
\end{theorem}

\begin{proof}
   Assume
   \[
   e_X e_Y e_X = e_Y e_X e_Y = e_X e_Y
   \]
   holds.
   Then 
   \[
   (e_X e_Y)^2 = e_X e_Y e_X e_Y = (e_X e_Y e_X) e_Y = e_X e_Y e_Y = e_X e_Y,
   \]
   which implies that $e_X e_Y$ is an idempotent, and hence by Theorem~\ref{ide-p}, $p_{\Theta}(X, Y)$ holds.

   Conversely, assume $p_{\Theta}(X, Y)$ holds.
   By Theorem~\ref{ide-p}, $e_X e_Y$ is an idempotent.
   By Proposition~\ref{hk_idempotents}, it follows that $X \cup Y \in \mathcal{A}_{\Theta}$ and 
   \begin{equation}\label{z}
      e_X e_Y = e_{X \cup Y}.
   \end{equation}
   Since $X, Y, X \cup Y \in \mathcal{A}_{\Theta}$, Lemma~\ref{lemma:join} implies that $e_{X \cup Y} e_X$ and $e_Y e_{X \cup Y}$ are idempotents as well.
   By Proposition~\ref{hk_idempotents}, we further have
   \[
   e_{X \cup Y} e_X = e_{(X \cup Y) \cup X} = e_{X \cup Y} \quad \text{and} \quad e_Y e_{X \cup Y} = e_{Y \cup (X \cup Y)} = e_{X \cup Y}.
   \]
   Using~\eqref{z}, we conclude that 
   \[
   e_X e_Y e_X = e_Y e_X e_Y = e_X e_Y.
   \]
\end{proof}

\begin{corollary}\label{cor:p}
   Let $X, Y \in \mathcal{A}_{\Theta}$.
   Then $p_{\Theta}(X, Y)$ and $p_{\Theta}(Y, X)$ both hold if and only if $e_X e_Y = e_Y e_X$.
\end{corollary}

\begin{proof}
   If $p_{\Theta}(X, Y)$ and $p_{\Theta}(Y, X)$ both hold, then by Theorem~\ref{theorem:p}, we have 
   \[
   e_X e_Y e_X = e_Y e_X e_Y = e_X e_Y \quad \text{and} \quad e_Y e_X e_Y = e_X e_Y e_X = e_Y e_X,
   \]
   which implies $e_X e_Y = e_Y e_X$.

   Conversely, if $e_X e_Y = e_Y e_X$, then 
   \[
   e_X e_Y e_X = e_X (e_Y e_X) = e_X e_X e_Y = e_X e_Y = e_Y e_X,
   \]
   and
   \[
   e_Y e_X e_Y = (e_Y e_X) e_Y = e_X e_Y e_Y = e_X e_Y = e_Y e_X,
   \]
   which implies 
   \[
   e_X e_Y e_X = e_Y e_X e_Y = e_X e_Y \quad \text{and} \quad e_Y e_X e_Y = e_X e_Y e_X = e_Y e_X.
   \]
   Theorem~\ref{theorem:p} implies that $p_{\Theta}(X, Y)$ and $p_{\Theta}(Y, X)$ both hold.
\end{proof}

\section{Homomorphisms from $\hkg$ to $\hk$}\label{HOM}

In this section, we study the set $\mathrm{Hom}(\hkg, \hk)$ of all monoid homomorphisms from $\hkg$ to $\hk$.

\begin{theorem}\label{theorem:p-pure}
    Let $\vp : \{y_1, y_2, \dots, y_m\} \to \hk$ be a map defined on the generators of $\hkg$.
    The map $\vp$ extends to a monoid homomorphism from $\hkg$ to $\hk$ if and only if there exist sets $X_1, X_2, \dots, X_m$ with
    $X_i \subseteq V_{\Theta}$ such that, for all $i, j \in \{1, 2, \dots, m\}$, the following conditions hold:
    \begin{enumerate}[label=(\roman*), ref=\roman*]
        \item\label{p1} $X_i \in \mathcal{A}_{\Theta}$;
        \item\label{p2} $\vp(y_i) = e_{X_i}$;
        \item\label{p3} $p_{\Theta}(X_i, X_j)$ and $p_{\Theta}(X_j, X_i)$ both hold if $y_i$ and $y_j$ are not connected in $\Gamma$;
        \item\label{p4} $p_{\Theta}(X_i, X_j)$ holds if $y_i \tog y_j$.
    \end{enumerate}
\end{theorem}

\begin{proof}
    Assume that the map $\vp$ extends to a monoid homomorphism from $\hkg$ to $\hk$, which we also denote by $\vp$.
    Since a monoid homomorphism maps idempotents to idempotents, $\vp(y_i)$ is an idempotent in $\hk$ for all $i \in \{1, 2, \dots, m\}$.
    By Proposition~\ref{proposition:idempotents}, $\vp(y_i)$ is of the form $e_X$ for some $X \in \mathcal{A}_{\Theta}$.
    Hence there exist sets $X_1, X_2, \dots, X_m$ with $X_i \in \mathcal{A}_{\Theta}$ for $i \in \{1, 2, \dots, m\}$, such that 
    \[
    \vp(y_i) = e_{X_i},
    \]
    for all $i \in \{1, 2, \dots, m\}$, which proves~\eqref{p1} and~\eqref{p2}.
    Now let $i, j \in \{1, 2, \dots, m\}$.
    
    \textbf{Case 1: $y_i$ and $y_j$ are not connected in $\Gamma$.}
    In this case, we have 
    \[
    y_i y_j = y_j y_i.
    \]
    Applying $\vp$ yields
    \[
    e_{X_i} e_{X_j} = e_{X_j} e_{X_i}.
    \]
    Corollary~\ref{cor:p} implies that $p_{\Theta}(X_i, X_j)$ and $p_{\Theta}(X_j, X_i)$ both hold. 
    This proves \eqref{p3}.

    \textbf{Case 2: $y_i \tog y_j$.}
    In this case, we have
    \[
    y_i y_j y_i = y_j y_i y_j = y_i y_j,
    \]
    and hence by applying $\vp$, we obtain
    \[
    e_{X_i} e_{X_j} e_{X_i} = e_{X_j} e_{X_i} e_{X_j} = e_{X_i} e_{X_j}.
    \]
    Theorem~\ref{theorem:p} implies that $p_{\Theta}(X_i, X_j)$ holds.
    This proves~\eqref{p4}.

    We have proved the necessity of the condition stated in the theorem.
    Now we prove sufficiency.
    We claim that the elements $\vp(y_1), \vp(y_2), \dots, \vp(y_m)$ satisfy the defining relations of $\hkg$.
    Hence, we need to show that $e_{X_1}, e_{X_2}, \dots, e_{X_m}$ satisfy the defining relations of $\hkg$.
    Each $e_{X_i}$ for $i \in \{1, 2, \dots, m\}$ is an idempotent by Proposition~\ref{proposition:idempotents}.
    Let $i, j \in \{1, 2, \dots, m\}$.

    \textbf{Case 1: $y_i$ and $y_j$ are not connected in $\Gamma$.}
    We need to show that
    \[
    e_{X_i} e_{X_j} = e_{X_j} e_{X_i}.
    \]
    By assumption, $p_{\Theta}(X_i, X_j)$ and $p_{\Theta}(X_j, X_i)$ both hold, and hence the 
    desired equality follows from Corollary~\ref{cor:p}.
    
    \textbf{Case 2: $y_i \tog y_j$.}
    We need to show that
    \[
    e_{X_i} e_{X_j} e_{X_i} = e_{X_j} e_{X_i} e_{X_j} = e_{X_i} e_{X_j}.
    \]
    By assumption, $p_{\Theta}(X_i, X_j)$ holds, and hence the desired equality follows from 
    Theorem~\ref{theorem:p}.
     
    Hence, the elements $\vp(y_1), \vp(y_2), \dots, \vp(y_m)$ satisfy the defining relations of $\hkg$, and thus we can extend $\vp$ to a monoid homomorphism from $\hkg$ to $\hk$.
\end{proof}

Denote by
\[
\mathcal{S}_{\Gamma, \Theta} = \{(X_1, X_2, \dots, X_m) \mid X_i \subseteq V_{\Theta} \text{ for } i \in \{1, 2, \dots, m\} \},
\]
the set of all sequences of length $m$ of subsets of $V_{\Theta}$.
\begin{definition}\label{definition:p-pure}
   A sequence $(X_1, X_2, \dots, X_m)$ of length $m$ of subsets of $V_{\Theta}$ is said to be $(\Gamma, \Theta)$\emph{-admissible} if, for all $i, j \in \{1, 2, \dots, m\}$, the following conditions hold: 
   \begin{enumerate}[label=(\roman*), ref=\roman*]
        \item $X_i \in \mathcal{A}_{\Theta}$;
        \item $p_{\Theta}(X_i, X_j)$ and $p_{\Theta}(X_j, X_i)$ both hold if $y_i$ and $y_j$ are not connected in $\Gamma$;
        \item $p_{\Theta}(X_i, X_j)$ holds if $y_i \tog y_j$.
    \end{enumerate}
\end{definition}

We also define
\[
\sgt = \{(X_1, X_2, \dots, X_m) \in \mathcal{S}_{\Gamma, \Theta} \mid (X_1, X_2, \dots, X_m) \text{ is $(\Gamma, \Theta)$-admissible} \},
\]
and the map $\Phi : \mathrm{Hom}(\hkg, \hk) \to \mathcal{S}_{\Gamma, \Theta}$ by 
\[
\Phi(\vp) = \big(c(\vp(y_1)), c(\vp(y_2)), \dots, c(\vp(y_m))\big).
\]

\begin{proposition}\label{phi:bijective}
   The map $\Phi : \mathrm{Hom}(\hkg, \hk) \to \sgt$ is a bijection.
\end{proposition}

\begin{proof}
   Let $\vp \in \mathrm{Hom}(\hkg, \hk)$.
   By Theorem~\ref{theorem:p-pure}, there exists a $(\Gamma, \Theta)$-admissible sequence $(X_1, X_2, \dots, X_m)$ such that
   \[
   \vp(y_i) = e_{X_i},
   \]
   for all $i \in \{1, 2, \dots, m\}$.
   Hence, 
   \[
   \Phi(\vp) = (c(\vp(y_1)), c(\vp(y_2)), \dots, c(\vp(y_m))) = (X_1, X_2, \dots, X_m),
   \]
   which proves that $\Phi$ maps into $\sgt$.

   Since every homomorphism from $\hkg$ to $\hk$ is uniquely determined by the images of the generators,
   since every homomorphism from $\hkg$ to $\hk$ maps generators to idempotents, and
   since idempotents in $\hk$ are uniquely determined by their content, $\Phi$ is injective.
   
   Let $(X_1, X_2, \dots, X_m) \in \sgt$.
   By Theorem~\ref{theorem:p-pure}, the map $y_i \mapsto e_{X_i}$ extends to a monoid homomorphism $\vp$ from $\hkg$ to $\hk$.
   Since $c(\vp(y_i)) = c(e_{X_i}) = X_i$ for $i \in \{1, 2, \dots, m\}$, we obtain $\Phi(\vp) = (X_1, X_2, \dots, X_m)$, and hence surjectivity follows.
\end{proof}

Now we denote by $\mathcal{B}_{n, m}$ the set of all $n \times m$ Boolean matrices, that is, $n \times m$ matrices with entries $0$ or $1$.
For $A \in \mathcal{B}_{n, m}$, $i \in \{1, 2, \dots, n\}$, and $j \in \{1, 2, \dots, m\}$, denote by $A_{i, j}$ the $(i, j)$ entry of $A$.
We define the map $\Psi : \mathcal{S}_{\Gamma, \Theta} \to \mathcal{B}_{n, m}$ by
\[
\Psi(X_1, X_2, \dots, X_m)_{i, j} = 
\begin{cases}
1, & x_i \in X_j, \\
0, & x_i \notin X_j,
\end{cases}
\]
for all $i \in \{1, 2, \dots, n\}$ and $j \in \{1, 2, \dots, m\}$.
The map $\Psi$ has a particularly simple interpretation; it encodes, in matrix form, which elements are contained in each set of the sequence.

\begin{proposition}\label{psi:bijective}
   The map $\Psi : \mathcal{S}_{\Gamma, \Theta} \to \mathcal{B}_{n, m}$ is a bijection.
\end{proposition}

\begin{proof}
   The proof is immediate from the definition of $\Psi$.
\end{proof}

\begin{theorem}\label{hom:main}
   The map $\Psi \circ \Phi : \mathrm{Hom}(\hkg, \hk) \to \Psi(\sgt)$ is a bijection.
\end{theorem}

\begin{proof}
   By Proposition~\ref{phi:bijective}, $\Phi : \mathrm{Hom}(\hkg, \hk) \to \sgt$ is a bijection.
   By Proposition~\ref{psi:bijective}, $\Psi$ is bijective, and since $\sgt \subseteq \mathcal{S}_{\Gamma, \Theta}$, it follows that $\Psi|_{\sgt} : \sgt \to \Psi(\sgt)$ is a bijection.
   Therefore, $\Psi \circ \Phi$ is a bijection from $\mathrm{Hom}(\hkg, \hk)$ to $\Psi(\sgt)$.
\end{proof}

\section{Endomorphisms of $\hk$}\label{ENDO}

Here, we study the endomorphism monoid $\mathrm{End}(\hk)$.
We use results from Section~\ref{HOM} in the special case when $\Gamma = \Theta$.
In this case, we have $m = n$ and $y_i = x_i$ for $i = 1, 2, \dots, n$.

The following definitions and statements are very similar to the results in~\cite{andrensek}.
We prove these results again because the results in~\cite{andrensek} were derived for the Kiselman monoid, which is a special case of $\hk$.
For completeness, we include the proofs here for $\hk$.

For $\mathbf{X} = (X_1, X_2, \dots, X_n) \in \mathcal{S}_{\Theta, \Theta}$, denote by $\mathbf{X}_i = X_i$ for $i \in \{1, 2, \dots, n\}$.
We define the binary operation $\ast$ on $\mathcal{S}_{\Theta, \Theta}$ as follows.
Let $(X_1, X_2, \dots, X_n)$ and $(Y_1, Y_2, \dots, Y_n)$ be elements of $\mathcal{S}_{\Theta, \Theta}$.
Then we set
\[
(X_1, X_2, \dots, X_n) \ast (Y_1, Y_2, \dots, Y_n) = (\bigcup_{x_j \in Y_1} X_j, \bigcup_{x_j \in Y_2} X_j, \dots, \bigcup_{x_j \in Y_n} X_j).
\]
It quickly follows that $(\mathcal{S}_{\Theta, \Theta}, \ast)$ is a monoid with unit element $(\{x_1\}, \{x_2\}, \dots, \{x_n\})$.

\begin{proposition}\label{phi:homo}
   The map $\Phi : \mathrm{End}(\hk) \to \mathcal{S}_{\Theta, \Theta}$ is a monoid homomorphism from $\mathrm{End}(\hk)$ to $(\mathcal{S}_{\Theta, \Theta}, \ast)$.
\end{proposition}

\begin{proof} 
   Let $\psi, \vp \in \mathrm{End}(\hk)$ and let $i \in \{1, 2, \dots, n\}$.
   By Lemma~\ref{content:lemma}, we have 
   \[
   (\Phi(\psi \circ \vp))_i = c((\psi \circ \vp)(x_i)) = \bigcup_{x_j \in c(\vp(x_i))} c(\psi(x_j)).
   \]
   We also have
   \begin{align*}
      \big(\Phi(\psi) \ast \Phi(\vp)\big)_i &= \big( (c(\psi(x_1)), \dots, c(\psi(x_n))) \ast (c(\vp(x_1)), \dots, c(\vp(x_n))) \big)_i \\
      &= \bigcup_{x_j \in c(\vp(x_i))} c(\psi(x_j)).
   \end{align*}
   We thus have
   \[
   \big(\Phi(\psi \circ \vp)\big)_i = \big(\Phi(\psi) \ast \Phi(\vp)\big)_i.
   \]
   It follows that
   \begin{align*}
      \Phi(\psi \circ \vp) = \Phi(\psi) \ast \Phi(\vp).
   \end{align*}
   Moreover, since the identity map is the unit element of $\mathrm{End}(\hk)$ and
   \begin{align*}
      \Phi(\mathrm{id}_{\hk}) = (\{x_1\}, \{x_2\}, \dots, \{x_n\}),
   \end{align*}
   we conclude that $\Phi$ is a monoid homomorphism. 
\end{proof}

We define the binary operation $\cdot$ on $\mathcal{B}_{n, n}$ as follows. 
Let $A, B \in \mathcal{B}_{n, n}$.
Then $C = A \cdot B$, with
\begin{align*}
   C_{i, j} = \bigvee_{k=1}^n (A_{i, k} \land B_{k, j}),
\end{align*}
for all $i, j \in \{1, 2, \dots, n\}$.
Here, $\lor$ denotes the Boolean join and $\land$ denotes the Boolean meet on $\{0, 1\}$.
One can easily check that $(\mathcal{B}_{n, n}, \cdot)$ is a monoid, and its unit element is the identity matrix $I$ of size $n \times n$.

\begin{proposition}\label{psi:homo}
   The map $\Psi : \mathcal{S}_{\Theta, \Theta} \to \mathcal{B}_{n, n}$ is a monoid homomorphism from $(\mathcal{S}_{\Theta, \Theta}, \ast)$ to $(\mathcal{B}_{n, n}, \cdot)$.
\end{proposition}

\begin{proof}
   Let $\mathbf{X}, \mathbf{Y} \in \mathcal{S}_{\Theta, \Theta}$, where $\mathbf{X} = (X_1, X_2, \dots, X_n)$ and $\mathbf{Y} = (Y_1, Y_2, \dots, Y_n)$.
   We compute
   \[
   \Psi(\mathbf{X} \ast \mathbf{Y}) = \Psi(\bigcup_{x_j \in Y_1} X_j, \bigcup_{x_j \in Y_2} X_j, \dots, \bigcup_{x_j \in Y_n} X_j).
   \]
   It follows that for $i, j \in \{1, 2, \dots, n\}$ we have 
   \begin{align*}
      \Psi(\mathbf{X} \ast \mathbf{Y})_{i, j} = 1 &\iff x_i \in \bigcup_{x_l \in Y_j} X_l \\
      &\iff \text{there exists } x_k \in Y_j \text{ such that } x_i \in X_k.
   \end{align*}
   On the other hand,
   \[
   (\Psi(\mathbf{X}) \cdot \Psi(\mathbf{Y}))_{i, j} = \bigvee_{k=1}^n (\Psi(\mathbf{X})_{i, k} \land \Psi(\mathbf{Y})_{k, j}).
   \]
   Hence, we have
   \begin{align*}
      (\Psi(\mathbf{X}) \cdot \Psi(\mathbf{Y}))_{i, j} = 1 &\iff \text{there exists } k \in \{1, 2, \dots, n\} \text{ such that } \Psi(\mathbf{X})_{i, k} = \Psi(\mathbf{Y})_{k, j} = 1 \\
      &\iff \text{there exists } k \in \{1, 2, \dots, n\} \text{ such that } x_i \in X_k \text{ and } x_k \in Y_j \\
      &\iff \text{there exists } x_k \in Y_j \text{ such that } x_i \in X_k.
   \end{align*}
   We thus obtain
   \[
   \Psi(\mathbf{X} \ast \mathbf{Y})_{i, j} = 1 \iff (\Psi(\mathbf{X}) \cdot \Psi(\mathbf{Y}))_{i, j} = 1,
   \]
   and since $\Psi(\mathbf{X} \ast \mathbf{Y})$ and $\Psi(\mathbf{X}) \cdot \Psi(\mathbf{Y})$ are both Boolean matrices, we get
   \[
   \Psi(\mathbf{X} \ast \mathbf{Y})_{i, j} = (\Psi(\mathbf{X}) \cdot \Psi(\mathbf{Y}))_{i, j}.
   \]
   Since $i$ and $j$ are arbitrary, it follows that
   \[
   \Psi(\mathbf{X} \ast \mathbf{Y}) = \Psi(\mathbf{X}) \cdot \Psi(\mathbf{Y}).
   \]
   Furthermore, since $\Psi(\{x_1\}, \{x_2\}, \dots, \{x_n\}) = I$, it follows that $\Psi$ is a monoid homomorphism.
\end{proof}

\begin{theorem}\label{phi:psi:homo}
   The map $\Psi \circ \Phi : \mathrm{End}(\hk) \to \mathcal{B}_{n, n}$ is a monoid homomorphism from $\mathrm{End}(\hk)$ to $(\mathcal{B}_{n, n}, \cdot)$.
\end{theorem}

\begin{proof}
   The proof is immediate from Propositions~\ref{phi:homo} and~\ref{psi:homo}, and the fact that a composition of monoid homomorphisms is a monoid homomorphism.
\end{proof}

\begin{theorem}\label{endo:main}
   The set $\Psi(\stt)$, equipped with the operation $\cdot$, is a monoid and we have
   \[
   \mathrm{End}(\hk) \cong (\Psi(\stt), \cdot),
   \]
   where the corresponding monoid isomorphism is given by $\Psi \circ \Phi$.
\end{theorem}

\begin{proof}
   By Theorem~\ref{hom:main}, $\Psi \circ \Phi : \mathrm{End}(\hk) \to \Psi(\stt)$ is a bijection.
   Pick any $A, B \in \Psi(\stt)$.
   Then there exist $\psi, \vp \in \mathrm{End}(\hk)$ such that 
   \[
   (\Psi \circ \Phi)(\psi) = A \quad \text{and} \quad (\Psi \circ \Phi)(\vp) = B.
   \]
   Using Theorem~\ref{phi:psi:homo}, we conclude that 
   \[
   A \cdot B = (\Psi \circ \Phi)(\psi) \cdot (\Psi \circ \Phi)(\vp) = (\Psi \circ \Phi)(\psi \circ \vp) \in \Psi(\stt),
   \]
   since $\psi \circ \vp \in \mathrm{End}(\hk)$.
   This shows that $\Psi(\stt)$ is closed under $\cdot$.
   Since
   \[
   I = (\Psi \circ \Phi)(\mathrm{id}_{\hk}) \in \Psi(\stt),
   \]
   it follows that $(\Psi(\stt), \cdot)$ is a submonoid of $(\mathcal{B}_{n, n}, \cdot)$, and
   \[
   \Psi \circ \Phi : \mathrm{End}(\hk) \to \Psi(\stt)
   \]
   is an isomorphism of monoids from $\mathrm{End}(\hk)$ to $(\Psi(\stt), \cdot)$ by Theorems~\ref{hom:main} and~\ref{phi:psi:homo}.
\end{proof}

\bibliography{sn-bibliography}

\end{document}